\documentclass[a4paper,10pt]{amsart}
\usepackage{color}
\usepackage[english]{babel}
\usepackage[hidelinks]{hyperref}
\usepackage{amsmath,amssymb,amsthm}
\usepackage{mathtools}
\usepackage{graphicx}

\DeclarePairedDelimiter{\ip}\langle\rangle

\makeatletter
\newcommand{\mypm}{\mathbin{\mathpalette\@mypm\relax}}
\newcommand{\@mypm}[2]{\ooalign{%
  \raisebox{.1\height}{$#1+$}\cr
  \smash{\raisebox{-.6\height}{$#1-$}}\cr}}
\makeatother

\makeatletter
\def\@xfootnote[#1]{%
  \protected@xdef\@thefnmark{#1}%
  \@footnotemark\@footnotetext}
\makeatother

\DeclareMathOperator{\Gram}{Gram}
\DeclareMathOperator{\rank}{rank}
\DeclareMathOperator{\Cor}{Cor}
\DeclareMathOperator{\Clifford}{\mathcal C}
\DeclareMathOperator{\hcpsd}{cpsd-rank_\C}
\newcommand{\CSP}{\mathcal{CS}_{\hspace{-0.065em}+}}
\DeclareMathOperator{\cpsd}{cpsd-rank}
\DeclareMathOperator{\scpsd}{cpsd-rank_\R}

\newcommand{\rkpsd}{\text{\rm rank}_{\text{\rm psd}}}
\newcommand{\fib}{\operatorname{fib}}

\newcommand{\CP}{\mathcal{CP}}
\newcommand{\cprank}{\text{\rm cp-rank}}

\renewcommand{\S}{\mathcal{S}_+}

\newcommand{\R}{\mathbb{R}}
\newcommand{\K}{\mathbb{K}}
\newcommand{\C}{\mathbb{C}}

\newcommand{\T}{{\sf T}}
\DeclareMathOperator{\Tr}{Tr}

\renewcommand{\min}{\mathrm{min}}

\renewcommand{\max}{\mathrm{max}}

\renewcommand{\epsilon}{\varepsilon}
\renewcommand\Re{\operatorname{Re}}
\renewcommand\Im{\operatorname{Im}}

\newcommand{\ME}{{\mathcal E}}
\DeclareMathOperator{\Span}{Span}
\newcommand{\Diag}{\text{\rm Diag}}

\newtheorem{defin}{Definition}[section]

\newtheorem{proposition}[defin]{Proposition}
\newtheorem{theorem}[defin]{Theorem}
\newtheorem*{theorem*}{Theorem 1.1}

\newtheorem{corollary}[defin]{Corollary}
\newtheorem{lemma}[defin]{Lemma}

\newtheorem{claim*}{Claim}
\newtheorem{example}[defin]{Example}
\newtheorem{question}[defin]{Question}
\newtheorem{algorithm}[defin]{Algorithm}

\title{Matrices with high completely positive semidefinite rank}

\author[Sander Gribling]{\lowercase{\href{https://sites.google.com/site/sandergribling/}}{Sander Gribling$^{1,3}$}}
\author[David de Laat]{\lowercase{\href{http://www.daviddelaat.nl}}{David de Laat$^{1,3}$}}
\author[Monique Laurent]{\lowercase{\href{http://homepages.cwi.nl/~monique}}{Monique Laurent$^{1,2}$}}

\date{October 17, 2016}



\begin{document}

\begin{abstract}
A real symmetric matrix $M$ is completely positive semidefinite if it admits a Gram representation by (Hermitian) positive semidefinite matrices of any~size~$d$. The smallest such $d$ is called the (complex) completely positive semidefinite rank of $M$, and it is an open question whether there exists an upper bound on this number as a function of the matrix size. We construct completely positive semidefinite matrices of size $4k^2+2k+2$ with complex~completely positive semidefinite rank $2^k$ for any positive integer $k$. This shows that if such an upper bound exists, it has to be at least exponential in the matrix size. For this we exploit connections to quantum information theory and we construct extremal bipartite correlation matrices of large rank. We also exhibit a class of completely positive matrices with quadratic (in terms of the matrix size) completely positive rank, but with linear completely positive semidefinite rank, and we make  a connection to the existence of Hadamard matrices.
\end{abstract}

\footnotetext[1]{CWI, Amsterdam, The Netherlands}
\footnotetext[2]{Tilburg University, Tilburg, The Netherlands}
\footnotetext[3]{The work was supported by the Netherlands Organization for Scientific Research, grant number 617.001.351.}

\maketitle

\section{Introduction}

A matrix is said to be \emph{completely positive semidefinite} if it admits a Gram representation by (Hermitian) positive semidefinite matrices of any size. The $n \times n$ completely positive semidefinite matrices form a convex cone, called the completely positive semidefinite cone, which is denoted by $\CSP^n$. 

The motivation for the study of the completely positive semidefinite cone is twofold. Firstly, the completely positive semidefinite cone $\CSP^n$ is a natural analog of the completely positive cone $\CP^n$, which consists of the matrices admitting a factorization by nonnegative vectors. The cone $\CP^n$ is well studied (see, for example, the monograph 
\cite{BSM03}), and, in particular, it can be used to model classical graph parameters. For instance, \cite{dKP02} shows how to model the stability number of a graph as a conic optimization problem over the completely positive cone. A second motivation lies in the  connection to quantum information theory. Indeed, the cone $\CSP^n$ was introduced in \cite{LP15} to model quantum graph parameters (including quantum stability numbers) as conic optimization problems, an approach extended in \cite{MR15} for quantum graph homomorphisms and in \cite{Antonios:2015} for quantum correlations.

In this paper we are interested in the size of the factors needed in Gram representations of matrices. This type of question is of interest for factorizations by nonnegative vectors as well as by (Hermitian) positive semidefinite matrices.

Throughout we use the following notation. 
For $X,Y\in \C^{d\times d}$,  $X^*$ is the conjugate transpose and $\langle X, Y\rangle = \mathrm{Tr}(X^*Y)$ is the trace inner product. For vectors $u,v\in \R^d$, $\langle u,v\rangle= u^Tv$ denotes their Euclidean inner product.

A matrix $M$ is said to be \emph{completely positive} if there exist nonnegative vectors $v_1,\ldots,v_n \in \R_+^d$ such that $M_{i,j} = \langle v_i, v_j \rangle$ for all $i,j\in [n]$. We call such a set of vectors a \emph{Gram representation} or \emph{factorization} of $M$ by nonnegative vectors. The smallest $d$ for which these vectors exist is denoted by $\cprank(M)$ and is called the \emph{completely positive rank} of $M$. 

Similarly, a matrix $M$ is called \emph{completely positive semidefinite} if there exist (real symmetric or complex Hermitian) positive semidefinite $d \times d$ matrices $X_1,\ldots,X_n$ such that $M_{i,j} = \langle X_i, X_j \rangle$ for all $i,j\in [n]$. We call such a set of matrices a \emph{Gram representation} or \emph{factorization} of $M$ by (Hermitian) positive semidefinite matrices. The smallest $d$ for which there exists a Gram representation of $M$ by Hermitian positive semidefinite $d \times d$ matrices is denoted by $\hcpsd(M)$, and the smallest $d$ for which these matrices can be taken to be real is denoted by $\scpsd(M)$. We call this the \emph{real/complex completely positive semidefinite rank} of $M$. If a matrix has a factorization by Hermitian positive semidefinite matrices, then it also has a factorization by real positive semidefinite matrices.  In fact, for every $M \in \CSP^n$, we have
\[
\hcpsd(M) \leq \scpsd(M) \leq 2 \hcpsd(M)
\]
(see Section~\ref{sec:some properties}).

By construction, we have the inclusions
\[
\CP^n\subseteq \CSP^n \subseteq \mathcal S^n_+\cap\R^{n\times n}_+,
\]
where $\mathcal S_+^n$ is the cone of (real) positive semidefinite $n \times n$ matrices. The three cones coincide for $n\le 4$ (since doubly nonnegative matrices of size $n\le 4$ are completely positive), but both inclusions are strict for $n\ge 5$ (see \cite{LP15} for details).

By Carath\'eodory's theorem, the completely positive rank of a matrix in $\CP^n$ is at most $\binom{n+1}{2}+1$. In \cite{Shaked-Monderer} the following stronger bound is given:
\begin{equation}\label{eqcprank}
\cprank(M) \leq {n+1 \choose 2}-4 \quad \text{for} \quad M \in \mathcal{CP}^n \quad \text{and} \quad n \geq 5,
\end{equation}
which is also not known to be tight. No upper bound (as a function of $n$) is known for the completely positive semidefinite rank of matrices in $\CSP^n$. It is not even known whether such a bound exists. A positive answer would have strong implications. It would imply that the cone $\CSP^n$ is closed. This, in turn, would imply that the set of quantum correlations is closed, since it can be seen as a projection of an affine slice of the completely positive semidefinite cone (see \cite{MR14,Antonios:2015}). Whether the set of quantum correlations is closed is an open question in quantum information theory.
In contrast, as an application of the upper bound \eqref{eqcprank}, the completely positive cone $\CP^n$ is easily seen to be closed. A description of the closure of the completely positive semidefinite cone in terms of factorizations by positive elements in von Neumann algebras can be found in \cite{STM:15}. Such factorizations were used to show a separation between the closure of $\CSP^n$ and the doubly nonnegative cone $\smash{\mathcal S^n_+\cap \R^{n\times n}_+}$ (see \cite{FW14,LP15}).

\medskip
In this paper we show that if an upper bound exists for the completely positive semidefinite rank of matrices in $\CSP^n$, then it needs to grow at least exponentially in the matrix size $n$. Our main result is the following:

\begin{theorem*}\label{theomain}
For each positive integer $k$, there exists a completely positive semidefinite matrix $M$ of size $4k^2 +2k +2$ with $\hcpsd(M) =  2^k$.
\end{theorem*}

The proof of this result relies on a connection with quantum information theory and geometric properties of (bipartite) correlation matrices.
We refer to the main text for the definitions of quantum and bipartite correlations. A first basic ingredient is the fact from \cite{Antonios:2015} that a quantum correlation $p$ can be realized in local dimension $d$ if and only if there exists a certain completely positive semidefinite matrix $M$ with $\hcpsd(M)$ at most $d$. Then, the key idea is to construct a class of quantum correlations $p$ that need large local dimension. 

The papers \cite{VertesiPal, Slofstra11, Ji13} each use different techniques to show the existence of different quantum correlations that require large local dimension. Our main contribution  is to provide a unified, explicit construction of the quantum correlations from \cite{VertesiPal} and \cite{Slofstra11}, which uses the seminal work of Tsirelson \cite{Tsirelson:87,Tsirelson} combined with convex geometry and recent insights from rigidity theory. In addition,  we also give an explicit proof of Tsirelson's bound (see Corollary~\ref{corbound}) and we show examples where the bound is tight.

More specifically, we construct such quantum correlations from bipartite correlation matrices. For this we use the classical results of Tsirelson \cite{Tsirelson:87,Tsirelson}, which characterize bipartite correlation matrices in terms of operator representations and, using Clifford algebras, 
 we relate  the rank of extremal bipartite correlations to the local dimension of their operator representations. In this way we reduce the problem to finding bipartite correlation matrices that are extreme points of the set of bipartite correlations and have large rank.

\medskip
For the completely  positive rank we have the quadratic upper bound \eqref{eqcprank}, and completely positive matrices have been constructed whose completely positive rank grows quadratically in the size of the matrix. This is the case, for instance, for the matrices
\[
M_k=\begin{pmatrix} I_k & {1\over k}J_k\cr {1\over k}J_k & I_k\end{pmatrix}\in \CP^{2k},
\]
where $\cprank(M_k)=k^2$. Here $I_k\in \mathcal S^k$ is the identity matrix and $J_k\in\mathcal S^k$ is the all-ones matrix. This leads to the natural question of how fast $\scpsd(M_k)$ and $\hcpsd(M_k)$ grow. As a second result we show that the completely positive semidefinite rank grows linearly for the matrices $M_k$, and we exhibit a link to the question of existence of Hadamard matrices. More precisely, we show that $\hcpsd(M_k) = k$ for all $k$, and  $\scpsd(M_k)=k$ if and only if there exists a real Hadamard matrix of order $k$. In particular, this shows that the real and complex completely positive semidefinite ranks can be different.

\medskip
The completely positive and completely positive semidefinite ranks are symmetric analogs of the nonnegative and  positive semidefinite ranks. Here the nonnegative rank, denoted $\mathrm{rank}_+(M)$, of a matrix $M \in \R_+^{m \times n}$, is the smallest integer $d$ for which there exist nonnegative vectors $\{u_i\}$ and $\{v_j\}$ in $\smash{\R_+^d}$ such that $M_{i,j} = \langle u_i, v_j \rangle$ for all $i$ and $j$, and the positive semidefinite rank, denoted $\mathrm{rank}_{\mathrm{psd}}(M)$, is the smallest $d$ for which there exist positive semidefinite matrices $\{X_i\}$ and $\{Y_j\}$ in $\mathcal S_+^d$ such that $M_{i,j} = \langle X_i, Y_j \rangle$ for all $i$ and $j$. These notions have many applications, in particular to communication complexity and for the study of efficient linear or semidefinite extensions of convex polyhedra (see \cite{Ya91,GPT13}). Unlike in the symmetric setting, in the asymmetric setting the following bounds, which show a linear regime, can easily be checked:
\[ \rkpsd(M) \leq \rank_+(M) \le  \min(m,n). \]
We refer to \cite{psdrank} and the references therein for a recent overview of results about the positive semidefinite rank.
 
\bigskip\noindent
\textbf{Organization of the paper.} 
In Section~\ref{sec:some properties} we first present some simple properties of the (real/complex) completely positive semidefinite rank, and then investigate its value for the matrices $M_k$, where we also show a link to the existence of Hadamard matrices. We also give a simple heuristic for finding approximate positive semidefinite factorizations.
The proof of our main  result in Theorem \ref{theomain} boils down to several key ingredients which we treat in the subsequent sections. 

In Section \ref{secCmn} we group old and new results about the set of  bipartite correlation matrices.
In particular, we give a geometric characterization of the extreme points, we revisit some conditions due to Tsirelson and links to rigidity theory, and we construct a class of extreme bipartite correlations with optimal parameters.

In Section \ref{secquantumcor} we recall some characterizations, due to Tsirelson,  of bipartite correlations in terms of operator representations. We also recall  connections  to Clifford algebras, and for  bipartite correlations that are extreme points we relate their rank to the dimension of their operator representations.

Finally in Section \ref{secfinal} we introduce quantum correlations and recall their link to completely positive semidefinite matrices. We show how to construct quantum correlations  from bipartite correlation matrices, and we prove the main theorem.

\bigskip\noindent
\textbf{Note.} 
Upon completion of this paper we learned of the recent independent work  \cite{PrakashSikoraVarvitsiotisWei}, where a class of matrices with exponential $\cpsd$ is also constructed. The key idea of using extremal bipartite correlation matrices having large rank is the same. Our construction uses bipartite correlation matrices with optimized parameters meeting Tsirelson's upper bound (\ref{eqrankC}) (see Corollary \ref{corbound}). As a consequence, our completely positive semidefinite matrices have the best ratio between $\cpsd$ and size that can be obtained using this technique. 

\section{Some properties of the completely positive semidefinite rank}
\label{sec:some properties}

In this section we consider the (complex) completely positive semidefinite rank of matrices in the completely positive cone. In particular, for a class of matrices, we show a quadratic separation in terms of the matrix size between the completely positive and completely positive semidefinite ranks. We also mention a simple heuristic for building completely positive semidefinite factorizations, which we have used to test several explicit examples.

We start by collecting some simple properties of the (complex) completely positive semidefinite rank. 
A first observation is that if a matrix $M$ admits a Gram representation by Hermitian positive semidefinite matrices of size $d$, then it also admits a Gram representation by real symmetric positive semidefinite matrices of size $2d$, and thus 
$M$ is completely positive semidefinite  with $\scpsd(M) \leq 2d$. This is based on the well-known fact that mapping a Hermitian $d \times d$ matrix $X$ to 
\[
\frac{1}{\sqrt{2}} \begin{pmatrix} \Re(X) & \Im(X)\\ \Im(X)^\T & \Re(X) \end{pmatrix} \in \mathcal S^{2d}
\]
is an isometry that preserves positive semidefiniteness. The (complex) completely positive semidefinite rank is subadditive; that is, for $A, B \in \CSP^n$ and $\K=\R$ or $\K=\C$, we have
\[
\cpsd_{\K}(A+B) \leq \cpsd_{\K}(A) + \cpsd_{\K}(B),
\]
which can be seen as follows: If $A$ is the Gram matrix of $X_1,\ldots,X_n \in \K^{k \times k}$ and $B$ is the Gram matrix of $Y_1,\ldots,Y_n \in \K^{r \times r}$, then $A+B$ is the Gram matrix of the matrices $X_1\oplus Y_1,\ldots,X_n \oplus Y_n \in \K^{(k+r) \times (k+r)}$. 

For $M \in \CSP^n$ we have the inequalities
\[
\binom{\scpsd(M)+1}{2} \geq \mathrm{rank}(M) \quad \text{and} \quad \hcpsd(M)^2 \geq \mathrm{rank}(M),
\]
since  a factorization of $M$ by  real symmetric (resp., complex Hermitian) positive semidefinite $r \times r$ matrices  yields another factorization of $M$ by real vectors of size ${r+1\choose 2}$ (resp.,  by real vectors of size $r^2$).

 Finally, the next lemma shows that  if the (complex) completely positive semidefinite rank of a matrix is high, then each factorization by (Hermitian) positive semidefinite matrices must contain at least one matrix with  high rank. 
\begin{lemma}
Let $M \in \CSP^n$. For each Gram representation of $M$ by (Hermitian) positive semidefinite matrices $X_1,\ldots,X_n\in \K^{d\times d}$, with $\K \in \{\R, \C\}$, we have 
\[
\cpsd_{\K}(M) \leq \rank(X_1+\ldots +X_n).
\]
\end{lemma}
\begin{proof}
Let $v_1,\ldots,v_{d-k}$ be an orthonormal basis of $\ker(X_1) \cap \ldots \cap \ker(X_n)$, and let $u_1,\ldots,u_k$ be an orthonormal basis of $(\ker(X_1) \cap \ldots \cap \ker(X_n))^\perp$. Let $U$ be the $d\times k$ matrix with columns $u_1,\ldots,u_k$, and let $V$ be the $d \times (d-k)$ matrix with columns $v_1,\ldots,v_{d-k}$, so that $P = \begin{pmatrix}
	U& \hspace{-0.5em}V \end{pmatrix}$ is an orthogonal matrix. 
	Set $Y_i = U^* X_i U\in \K^{k\times k}$ for $i\in [n]$. Then $Y_i$ is (Hermitian) positive semidefinite. 
Since $X_iV=0$ by construction, we have
\[
\big\langle Y_i, Y_j \big\rangle = \big\langle U^* X_i U, U^* X_j U \big\rangle = 
 \big\langle P^* X_i P, P^* X_j P \big\rangle=
 \big\langle X_i,X_j\big\rangle
 \]
for all $i,j \in [n]$, which shows $M = \mathrm{Gram}(Y_1,\ldots,Y_n)$.
We have 
\[
\cpsd_{\mathbb K}(M) \leq k = n - \dim(\ker(X_1) \cap \ldots \cap \ker(X_n)),
\]
and because the matrices $X_1,\ldots,X_n$ are positive semidefinite, the right hand side is equal to
\[
n - \dim(\ker(X_1 + \ldots + X_n)) = \rank(X_1 + \ldots + X_n).\qedhere
\]
\end{proof}

\subsection{A connection to the existence of Hadamard matrices}
\label{sec:hadamard}

Consider the $2k \times 2k$ matrix 
\[ 
M_k = \begin{pmatrix} I_k & \frac{1}{k} J_k \\ \frac{1}{k} J_k & I_k \end{pmatrix},
\]
where $I_k$ is the $k \times k$ identity matrix and $J_k$ the $k \times k$ all-ones matrix.
The completely positive rank of $M_k$ equals $k^2$, which is well known and easy to check  (see Proposition \ref{cpranknsquared} below). This means the completely positive rank of these matrices is within a constant factor of the upper bound \eqref{eqcprank}. The significance of the matrices $M_k$ stems from the recently disproved (see \cite{Bomze, Bomze2}) Drew-Johnson-Loewy conjecture \cite{DJL94}. This conjecture states that $\lfloor n^2/4 \rfloor$ is an upper bound on the completely positive rank of $n \times n$ matrices, which means the matrices $M_k$ are sharp for this bound.

It is therefore natural to ask whether the matrices $M_k$ also have large (quadratic in $k$) completely positive semidefinite rank. As we see below  this is not the case. We show that the complex completely positive semidefinite rank is $k$, and we show that the real completely positive semidefinite rank  is equal to $k$ if and only if a real Hadamard matrix of order $k$ exists, which suggests that determining the completely positive semidefinite rank is a difficult problem in general.

A real (complex) {Hadamard} matrix of order $k$ is a $k\times k$ matrix with  pairwise orthogonal columns and whose entries are $\pm 1$-valued (complex valued with unit modulus).  A complex Hadamard matrix exists for any order; take for example
\begin{equation}\label{eqHk}
(H_k)_{i, j} = e^{2 \pi \mathbf{i}(i-1)(j-1)/k} \quad \text{ for } \quad i,j\in [k].
\end{equation}
On the other hand, it is still an open conjecture whether a real Hadamard matrix exists for each order $k$ that is a multiple of 4.

For completeness we first give a proof that the completely positive rank is $k^2$. Here, the support of  a vector $u\in\R^d$ is the set of indices $i \in [d]$ for which $u_i\ne 0$.

\begin{proposition} \label{cpranknsquared}
The completely positive rank of $M_k$ is equal to $k^2$.
\end{proposition}
\begin{proof}
For $i\in [k]$ consider the vectors 
 $v_i = \smash{1/\sqrt{k}} \, e_i \otimes \mathbf{1}$ and $u_i = \smash{1/\sqrt{k}}\, \mathbf{1} \otimes e_i$, where $e_i$ is the $i$th basis vector in $\R^k$ and $\mathbf{1}$ is the all-ones vector in $\R^k$. The vectors $v_1,\ldots,v_k, u_1,\ldots,u_k$ are nonnegative  and form a Gram representation of $M_k$, which shows $\cprank(M_k)\le k^2$.

Suppose $M_k = \text{Gram}(v_1, v_2, \ldots, v_k, u_1, u_2, \ldots, u_k)$ with $v_i,u_i \in \R^d_+$. In the remainder of the proof we show $d\ge k^2$. We have $(M_k)_{i,j} = \delta_{ij}$ for $1 \leq i,j \leq k$. Since the vectors $v_i$ are nonnegative, they must have disjoint supports. The same holds for the vectors $u_1,\ldots,u_k$. Since $(M_k)_{i,j} = 1/k > 0$ for $1 \leq i \leq k$ and $k+1 \leq j \leq 2k$, the support of $v_i$ overlaps with the support of $u_j$ for each $i$ and $j$. This means that for each $i \in [k]$, the size of the support of the vector $v_i$ is at least $k$. This is only possible if $d \geq k^2$. 
\end{proof}

\begin{proposition}
For each $k$ we have
$
\hcpsd(M_k) = k.
$
Moreover, we have $\scpsd(M_k)=k$ if and only if there exists a real Hadamard matrix of order $k$.
\end{proposition} 
\begin{proof}
The lower bound $\hcpsd(M_k) \geq k$ follows because $I_k$ is a principal submatrix of $M_k$ and $\hcpsd(I_k) = k$. To show $\hcpsd(M_k) \leq k$, we give a factorization by Hermitian positive semidefinite $k \times k$ matrices. For this consider the complex 
Hadamard matrix  $H_k$ in (\ref{eqHk}) and define the factors 
\[X_i = e_i e_i^\T \quad \text{and} \quad Y_i= \frac{u_i u_i^*}{k} \quad \text{for} \quad  i \in [k],\] where $e_i$ is the $i$th standard basis vector of $\R^k$ and $u_i$ is the $i$th column of $H_k$. By direct computation it follows that $M_k = \Gram(X_1,\ldots,X_{k}, Y_1,\ldots, Y_k)$. 

We now show that $\scpsd(M_k)=k$  if and only if there exists a real Hadamard matrix of order $k$. One direction follows directly from the above proof: If a real Hadamard matrix of size $k$ exists, then we can replace $H_k$ by this real matrix and this yields a factorization by real positive semidefinite $k \times k$ matrices. 

Now assume $\scpsd(M_k) = k$ and let $X_1, \ldots, X_k, Y_1, \ldots, Y_k \in \mathcal S_+^k$ be a Gram representation of $M$.
We first show there exist two orthonormal bases $u_1, \ldots, u_k$ and $v_1, \ldots, v_k$ of $\R^k$ such that  $X_i = u_i u_i^\T$ and $Y_i = v_i v_i^\T$. For this we observe that $I = \Gram(X_1,\ldots,X_k)$, which implies $X_i \neq 0$ and $X_i X_j = 0$ for all $i \neq j$. Hence, for all $i \neq j$, the range of $X_j$ is contained in the kernel of $X_i$. Therefore the range of $X_i$ is orthogonal to the range of $X_j$. We now have $\smash{\sum_{i=1}^k} \mathrm{dim}(\mathrm{range}(X_i)) \leq k$ and $\mathrm{dim}(\mathrm{range}(X_i)) \geq 1$ for all $i$. From this it follows that $\rank(X_i) = 1$ for all $i \in [k]$. This means there exist $u_1, \ldots, u_k \in \R^k$ such that $X_i = u_i u_i^\T$ for all $i$. From $I = \Gram(X_1,\ldots,X_k)$ it follows that the vectors $u_1,\ldots,u_k$ form an orthonormal basis of $\R^k$. The same argument can be made for the matrices $Y_i$, thus $Y_i=v_iv_i^\T$ and the vectors $v_1,\ldots,v_k$ form an orthonormal basis of $\R^k$.
Up to an orthogonal transformation we may assume that the first basis is the standard basis; that is, $u_i = e_i$ for $i \in [k]$. We then obtain
\[
\frac{1}{k} = (M_k)_{i,j+k} = \ip{e_i,v_j}^2 = \big((v_j)_i\big)^2 \quad \text{for} \quad i,j \in [k],
\]
hence $(v_j)_i = \pm 1/\sqrt{k}$. Therefore, the $k \times k$ matrix whose $k$th column is $\sqrt{k}\, v_k$ is a real Hadamard matrix.
\end{proof}

The above proposition leaves open the exact determination of $\scpsd(M_k)$ for the cases where a real Hadamard matrix of order $k$ does not exist. Extensive experimentation using the heuristic from Section~\ref{sec:heuristic} suggests that for $k = 3, 5, 6, 7$ the real completely positive semidefinite rank of $M_k$ equals $2k$, which leads to the following question:

\begin{question}
Is the real completely positive semidefinite rank of $M_k$ equal to $2k$ if a real Hadamard matrix of size $k \times k$ does not exist?
\end{question}

We also used the heuristic from Section~\ref{sec:heuristic} to check numerically that the aforementioned matrices from \cite{Bomze}, which have completely positive rank greater than $\lfloor n^2/4 \rfloor$, have small (smaller than $n$)  real completely positive semidefinite rank. In fact, in our numerical experiments we never found a completely positive $n \times n$ matrix for which we could not find a factorization in dimension $n$, which leads to the following question:

\begin{question}
Is the real (or complex) completely positive semidefinite rank of a completely positive $n \times n$ matrix upper bounded by $n$?
\end{question}

\subsection{A heuristic for finding Gram representations}
\label{sec:heuristic}

In this section we give an adaptation to the symmetric setting of the seesaw method from \cite{WernerWolf2001}, which is used to find good quantum strategies for nonlocal games. Given a matrix $M \in \CSP^n$ with $\scpsd(M) \leq d$, we give a heuristic to find a Gram representation of $M$ by positive semidefinite $d \times d$ matrices. Although this heuristic is not guaranteed to converge to a factorization of $M$, for small $n$ and $d$ (say, $n,d \leq 10$) it works well in practice by restarting the algorithm several times. The following algorithm seeks to minimize the function
\[
E(X_1,\ldots,X_n) =\underset{i,j \in [n]}{\max} |\langle X_i, X_j \rangle - M_{i,j}|.
\]

\begin{algorithm}
Initialize the algorithm by setting $k=1$ and generating random matrices $X_1^{0},\ldots,X_n^{0} \in \mathcal S_+^d$ that satisfy $\langle X_i^{0}, X_i^{0} \rangle = M_{i,i}$ for all $i \in [n]$. Iterate the following steps:
\begin{enumerate}
\item Let $(\delta,Y_1,\ldots,Y_n)$ be a (near) optimal solution of the semidefinite program
\[
\min \Big\{ \delta : \delta \in \R_+, \, Y_1,\ldots,Y_n \in \mathcal S_+^d, \, \Big| \big\langle X_i^{k-1}, Y_j \big\rangle - M_{i,j} \Big| \leq \delta \text{ for } i,j \in [n] \Big\}.
\]
\item Perform a line search to find the scalar $r \in [0, 1]$ minimizing
\[
E\big( (1-r) X_1^{k-1} + r Y_1, \ldots,  (1-r) X_n^{k-1} + r Y_n\big),
\]
and set $X_i^k = (1-r) X_i^{k-1} + r Y_i$ for each $i \in [n]$.
\item If $E(X_1^k,\ldots,X_n^k)$ is not small enough, increase $k$ by one and go to step (1). Otherwise, return the matrices $X_1^k$, $\ldots$, $X_n^k$.
\end{enumerate}
\end{algorithm}

\section{The set of bipartite correlations}\label{secCmn}

In this section we define the set $\Cor(m,n)$ of bipartite correlations and we discuss properties of the extreme points of $\Cor(m,n)$, which will play a crucial role in the construction of $\CSP$-matrices with large complex completely positive semidefinite rank. 
In particular we give a characterization of the extreme points of $\Cor(m,n)$ in terms of extreme points of the related set $\ME_{m+n}$ of correlation matrices. We use it to give a simple construction of a class of extreme points of $\Cor(m,n)$ with rank $r$, when $m=n={r+1\choose 2}$.
We also revisit conditions for extreme points introduced by Tsirelson \cite{Tsirelson:87} and point out links with universal rigidity. Based on these we can construct  extreme points of $\Cor(m,n)$ with rank $r$ when $m=r$ and $n={r\choose 2}+1$, which are used to prove our main result (Theorem~\ref{theomain}).

\subsection{Bipartite correlations  and correlation matrices}
A matrix $C \in \R^{m \times n}$ is called a {\em bipartite correlation matrix} if there exist real unit vectors  $x_1,\ldots,x_m,$ $y_1,\ldots,y_n\in \R^d$ (for some $d\ge 1$) such that 
$C_{s,t} = \langle x_s, y_t \rangle$ for all $s \in [m]$ and $t \in [n]$. Following Tsirelson \cite{Tsirelson:87}, any such system of real unit vectors 
is called a {\em $C$-system}. We let $\Cor(m, n)$ denote the set of all $m\times n$ bipartite correlation matrices. 

The \emph{elliptope} $\ME_n$  is defined as
\[
\mathcal{E}_n = \Big\{ E \in \S^n: E_{ii} = 1 \text{ for } i =1, \ldots, n\Big\},
\]
its elements are the {\em correlation matrices}, which can alternatively be defined as all matrices of the form $(\langle z_i,z_j\rangle)_{i,j=1}^n$ for some real unit vectors $z_1,\ldots,z_n\in \R^d$ ($d\ge 1$).
We have the surjective projection 
\begin{equation}\label{eqpi}
\pi \colon \mathcal E_{m+n} \to \Cor(m,n), \, \begin{pmatrix} Q & C \\ C^\T & R \end{pmatrix} \mapsto C. 
\end{equation}
Hence, $\Cor(m,n)$ is a projection of the elliptope $\ME_{m+n}$ and therefore a convex set. Given   $C \in \Cor(m,n)$, any matrix $E \in \mathcal E_{m+n}$ such that $\pi(E)=C$ is called an \emph{extension} of $C$ to the elliptope and we let $\fib(C)$ denote the fiber (the set of extensions) of $C$. 

Theorem~\ref{theoExCor} below characterizes  extreme points of $\Cor(m,n)$ in terms of extreme points of  $\mathcal{E}_{m+n}$. 
It is based on two intermediary results.  The first result (whose proof is easy) relates  extreme points $C\in\Cor(m,n)$ to properties of their set  of extensions $\fib(C)$. It is shown in \cite{ELV14} in a more general setting.

\begin{lemma}[{\cite[Lemma 2.4]{ELV14}}] \label{lemELV}
Let $C\in \Cor(m,n)$. Then  $C$ is an extreme point of $\Cor(m,n)$ if and only if the set $\fib(C)$ is a face of $\mathcal E_{m+n}$. Moreover, if $C$ is an extreme point of $\Cor(m,n)$, then every extreme point of $\fib(C)$ is an extreme point of $\mathcal E_{m+n}$.
\end{lemma}

The second result  (from  Tsirelson \cite{Tsirelson:87}) shows that  every extreme point $C$ of $\Cor(m,n)$ has a unique extension $E$ in $\mathcal E_{m+n}$, we give a proof for completeness.

\begin{lemma}[\cite{Tsirelson:87}] \label{lemExVec}
Assume $C$ is an extreme point of $\Cor(m,n)$. 
\begin{itemize}
\item[(i)] If $x_1,\ldots,x_m,y_1,\ldots,y_n$ is a $C$-system, then 
\[
\Span\{x_1,\ldots,x_m\}=\Span\{y_1,\ldots,y_n\}.
\]
\item[(ii)] The matrix $C$ has a unique extension to a matrix $E \in \mathcal E_{m+n}$, and there exists a $C$-system $x_1,\ldots,x_m,y_1,\ldots,y_n \in \R^r$, with $r = \rank(C)$, such that
\[
E = \Gram(x_1,\ldots,x_m,y_1,\ldots,y_n).
\]
\end{itemize}
\end{lemma}

\begin{proof}
We will use the following observation: Each matrix $C=(\langle a_s,b_t \rangle)_{s\in [m],t\in [n]}$, where $a_s,b_t$ are vectors with $\|a_s\|,\|b_t\|\le 1$, belongs to $\Cor(m,n)$ since it satisfies 
\[
C_{s,t} = \left\langle \begin{pmatrix}a_s\\ \sqrt{1-\|a_s\|^2}\\ 0 \end{pmatrix}, \begin{pmatrix} b_t\\ 0\\ \sqrt {1-\|b_t\|^2}\end{pmatrix} \right\rangle\ \ \text{ for all }  (s,t)\in [m]\times [n].
\]
 
(i) Set $V=\Span\{x_1,\ldots,x_m\}$ and assume  $y_k \not\in V$ for some $k\in [n]$. Let
$w$ denote the orthogonal projection of $y_k$ onto $V$. Then $\|w\|<1$ and one can choose a nonzero vector $u \in V$ such that $\|w \mypm u\| \le 1$. Define the matrices $C^{\mypm} \in \R^{m \times n}$ by
\[
C_{s,t}^{\mypm} = 
\begin{cases}
\langle x_s, w \mypm u \rangle & \text{if } t = k,\\
\langle x_s, y_t \rangle & \text{if } t \neq k.
\end{cases}
\]
Then, $C^{\mypm} \in \Cor(m,n)$ (by the above observation) and $C=(C^++C^-)/2$. As $C$ is an extreme point of $\Cor(m,n)$ one must have $C=C^+=C^-$. Hence $u$ is orthogonal to each $x_s$ and thus $u=0$, a contradiction. This shows the inclusion 
$\Span\{y_1,\ldots,y_m\} \subseteq \Span\{x_1,\ldots,x_m\}$  and the reverse  one follows in the same way.
 
(ii) Assume $\{x_s',y_t'\}$ and $\{x_s'',y_t''\}$ are two $C$-systems. We show $\langle x'_r,x_s'\rangle = \langle x_r'',x_s''\rangle$ for all $r,s\in S$ and 
 $\langle y_t',y_u'\rangle=\langle y_t'',y_u''\rangle$ for all $t,u\in T$.
For this define the vectors
 \[
 x_s = \frac{x'_s \oplus x_s''}{\sqrt{2}} \quad \text{and} \quad y_t= \frac{y_t'\oplus y_t''}{\sqrt{2}},
 \]
  which again form a $C$-system. Using (i), for any $s\in S$, there exist scalars $\lambda_t^{s}$ such that $x_s=\sum_{t\in T} \lambda^{s}_t y_t$ and thus 
 $x_s'=\sum_{t\in T} \lambda ^{s}_t y_t'$ and $x_s''=\sum_{t \in T} \lambda ^{s}_t y_t''$. This shows
\[
\langle x_r',x_s'\rangle= \sum_{t \in T} \lambda^{r}_t \langle y_t', x_s'\rangle = \sum_{t \in T}\lambda^{r}_t C_{s,t}=
 \sum_{t \in T}\lambda^{r}_t\langle y_t'',x_s''\rangle = \langle x_r'',x_s''\rangle
 \]
for all $r,s \in S$. The analogous argument shows $\langle y_t',y_u'\rangle=\langle y_t'',y_u''\rangle$ for all $t,u \in T$. This shows $C$ has a unique extension to a matrix $E \in \ME_{m+n}$.

Finally,  we show that   $\rank (E)=\rank (C)$. 
Say $E$ is the Gram matrix of $x_1,\ldots,x_m,y_1,\ldots,y_n$.
In view of (i), $\rank(E)=\rank\{x_1,\ldots,x_m\}$ and thus  it suffices to show that 
$\rank\{x_1,\ldots,x_m\}\le \rank (C)$.
For this note that if $\{x_s: s\in I\}$ (for some $I\subseteq S$) is linearly independent then the corresponding rows of $C$ are linearly independent, since 
$\sum_{s\in I}\lambda_s \langle x_s,y_t\rangle =0$ (for all $t\in T$) implies  $\sum_{s\in I}\lambda_s  x_s=0$ (using (i)) and thus  $\lambda_s=0$ for all $s$.  
\end{proof}

\begin{theorem} \label{theoExCor}
A matrix $C$ is an extreme point of $\Cor(m,n)$ if and only if  $C$ has a unique extension to a matrix $E\in \ME_{m+n}$ and $E$ is an extreme point of $\ME_{m+n}$.
\end{theorem}

\begin{proof}
Direct application of Lemma \ref{lemELV} and Lemma \ref{lemExVec} (ii).
\end{proof}

We can use the following lemma to construct explicit examples of extreme points of $\Cor(m,n)$ for the case $m=n$.

\begin{lemma} \label{lemElliptopeCor}
Each extreme point of $\mathcal{E}_n$ is an extreme point of $\Cor(n,n)$.
\end{lemma}

\begin{proof}
Let $C$ be an extreme point of $\mathcal{E}_n$. Define the matrix 
\[
E = \begin{pmatrix} C & C \\ C & C \end{pmatrix}.
\]
Then $E \in \mathcal{E}_{2n}$ is an extension of $C$. In view of Theorem \ref{theoExCor} it suffices to show that $E$ is the  unique extension of $C$ and that $E$ is an extreme point of $\mathcal{E}_{2n}$. With  $e_1,\ldots,e_n$ denoting the standard unit vectors in $\R^n$, observe that the vectors $e_i\oplus -e_i$ ($i\in [n]$) lie in the kernel of any matrix $E'\in \fib(C)$. Indeed, since $E'$ and $C$ have an all-ones diagonal we have
\[
(e_i \oplus - e_i)^\T E' (e_i \oplus -e_i) = 0,
\]
and since $E'$ is positive semidefinite this implies that $e_i \oplus -e_i \in \ker(E')$. This implies that $\fib(C)=\{E\}$. 
We now show that $E$ is an extreme point of $\mathcal{E}_{2n}$. For this let  $E_1, E_2 \in \mathcal{E}_{2n}$ and $0 < \lambda < 1$ such that $E = \lambda E_1 + (1-\lambda) E_2$. 
As $E_1,E_2$ are positive semidefinite, the kernel of $E$ is the intersection of the kernels of $E_1$ and $E_2$.
Hence the vectors $e_i\oplus -e_i$ belong to the kernels of $E_1$ and $E_2$ and thus
\[
E_1 = \begin{pmatrix} C_1 & C_1 \\ C_1 & C_1 \end{pmatrix} \quad \text{and} \quad E_2 = \begin{pmatrix} C_2 & C_2 \\ C_2 & C_2 \end{pmatrix}
\]
for some $C_1,C_2\in \ME_n$. Hence, $C=\lambda C_1+(1-\lambda) C_2$, which implies $C=C_1=C_2$, since $C$ is an extreme point of $\ME_n$. Thus $E=E_1=E_2$, which completes the proof.
\end{proof}

The above lemma shows how to construct extreme points of $\Cor(n,n)$ from extreme points of the elliptope $\mathcal E_n$. Li and Tam \cite{LiTam} give the following characterization of the extreme points of $\ME_n$.

\begin{theorem}[\cite{LiTam}] \label{theoLiTam}
Consider a matrix $E\in \ME_n$ with rank $r$ and unit vectors $z_1,\ldots,z_n\in \R^r$ such that $E=\Gram(z_1,\ldots,z_n)$. Then $E$ is an extreme point of $\ME_n$  if and only if
\begin{equation}\label{eqrkE}
\binom{r+1}{2} = \dim(\Span\{z_1z_1^\T ,\ldots,z_nz_n^\T \}).
\end{equation}
In particular, if $E$ is an extreme point of $\ME_n$, then ${r+1\choose 2}\le n$.
\end{theorem}

\begin{example} [\cite{LiTam}] \label{LiTamLem}
For each integer $r\ge 1$ there exists an extreme point of $\mathcal E_n$ of rank $r$, where  $n = \binom{r+1}{2}$. For example, let $e_1, \ldots, e_r$ be the standard basis vectors of $\R^r$ and define $$E = \mathrm{Gram}\Big(e_1, \ldots, e_r, \frac{e_1 + e_2}{\sqrt{2}}, \frac{e_1 + e_3}{\sqrt{2}}, \ldots, \frac{e_{r-1} + e_r}{\sqrt{2}}\Big).
$$ Then $E$ is an extreme point of $\ME_n$ of rank $r$.
\end{example}

Note that the above example is optimal in the sense that a rank $r$ extreme point of $\ME_n$ can  exist only if $n \geq \smash{\binom{r+1}{2}}$ (by Theorem~\ref{theoLiTam}). By combining this with Lemma~\ref{lemElliptopeCor}, this gives a class of extreme points  of $\Cor(m,n)$ with rank $r$ and $m=n={r+1\choose 2}$.

\subsection{Tsirelson's bound}
If $C$ is an extreme point of $\Cor(m,n)$ with rank $r$, then by Theorems~\ref{theoExCor} and \ref{theoLiTam} we have ${r+1\choose 2}\le m+n$. Tsirelson~\cite{Tsirelson} claimed  the stronger bound ${r+1\choose 2}\le m+n-1$ (see Corollary~\ref{corbound} below). 
In the rest of this section we  show how to derive this stronger bound of Tsirelson (which is given in \cite{Tsirelson} without proof). In the next section, we construct two classes of extreme bipartite correlation matrices, of which one meets Tsirelson's bound. To show Tsirelson's bound  we need to investigate in more detail the unique extension property for extreme points of $\Cor(m,n)$. 

Let $C\in \Cor(m,n)$ with rank $r$, let $\{x_s\}$, $\{y_t\}$ be a $C$-system in $\R^r$, and let 
\[
E=\Gram(x_1,\ldots,x_m,y_1,\ldots,y_n)\in\ME_{m+n}.
\]
In view of Theorem~\ref{theoExCor}, if $C$ is an extreme point of $\Cor(m,n)$,
then $E$ is the unique extension of $C$ in $ \ME_{m+n}$.
This uniqueness property can be rephrased as the requirement that an associated semidefinite program has a unique solution. Namely, consider the following dual pair of semidefinite programs:
\begin{equation}\label{eqsdpP}
\max \Big\{ 0 : X \in \mathcal S_+^{S \cup T}, \, X_{k,k}=1 \text{ for } k\in S \cup T, \, X_{s,t} = C_{s,t} \text{ for } s\in S,t\in T \Big\},
\end{equation}
\begin{equation}\label{eqsdpD}
\min \Big\{ \sum_{s\in S} \lambda_s +\sum_{t\in T}\mu_t +2\sum_{s\in S,t\in T}W_{s,t} C_{s,t} : \Omega=\begin{pmatrix}\Diag(\lambda) & W \cr W^\T  & \Diag(\mu) \end{pmatrix} \in \mathcal S_+^{S \cup T} \Big\}.
\end{equation} 
The feasible region of  problem (\ref{eqsdpP}) consists of all possible extensions of $C$ in $\ME_{m+n}$, and the feasible region of  (\ref{eqsdpD}) consists of  the positive semidefinite matrices $\Omega$ whose support  (consisting of all off-diagonal pairs $(i,j)$ with $\Omega_{i,j}\ne 0$) is contained in the complete bipartite graph with bipartition $S\cup T$. Moreover, 
the optimal values of both problems are equal to $0$. Finally, for any primal feasible (optimal) $X$ and dual optimal $\Omega$, equality $\Omega X=0$ holds, which implies that $\rank(X)+\rank(\Omega) \le m+n$. 

 Theorem \ref{theoLV} below (shown in \cite{LV14} in the more general context of universal rigidity) shows that if equality $\rank(X)+\rank(\Omega)=m+n$ holds (also known as \emph{strict complementarity}), then $X$ is in fact the {\em unique} feasible solution of program (\ref{eqsdpP}), and thus  $C$ has a {\em unique} extension in $\ME_{m+n}$.

\begin{theorem}
\label{theoLV}
Let $C\in \Cor(m,n)$ and let $\{x_s\}$, $\{y_t\}$ be a $C$-system spanning $\R^r$. Assume $E=\Gram(x_1,\ldots,x_m,y_1,\ldots,y_n)$ is an extreme point of $\ME_{m+n}$. If there exists an optimal solution $\Omega$ of program (\ref{eqsdpD}) with $\rank(\Omega)=m+n-r$, then $E$ is the only extension of $C$ in $\ME_{m+n}$.
\end{theorem}
\begin{proof}
Apply  \cite[Theorem 3.2]{LV14} to the bar framework $G(\mathbf p)$, where $G$ is the complete bipartite graph $K_{m,n}$ with bipartition $S\cup T$ and   $\mathbf p=\{x_s (s\in S), y_t (t\in T)\}$. Note that  the conditions  (v), (vi) in \cite[Theorem 3.2]{LV14} follow from $\Omega E=0$ and  the fact that $\{x_s\}, \{y_t\}\subset \R^r$ are C-systems spanning $\R^r$.
\end{proof}

In addition one can relate uniqueness of an extension of $C$ in the elliptope to the existence of a quadric separating the two point sets $\{x_s\}$ and $\{y_t\}$ (Theorem~\ref{theoExCorTsi} below). Roughly speaking, such a quadric allows us to construct a suitable optimal dual solution $\Omega$ and to apply Theorem \ref{theoLV}.  This property was  stated by Tsirelson \cite{Tsirelson}, however without proof. Interestingly, an analogous result was shown recently by Connelly and Gortler \cite{CG} in the setting of universal rigidity. We will give a sketch of a proof for Theorem~\ref{theoExCorTsi}. For this we use Theorem~\ref{theoLV}, arguments in \cite{CG}, and the following basic property of semidefinite programs (which can be seen as an analog of Farkas' lemma for linear programs).

\begin{lemma} \label{theoFarkas}
Given $A_1,\ldots,A_m \in \mathcal S^n$ and $b \in \R^m$, and assume that there exists a matrix $X \in \mathcal S^n$ such that $\langle A_j, X \rangle = b_j$ for all $j \in [m]$.
Then exactly one of the following two alternatives holds:
\begin{itemize}
\item[(i)]
There exists a matrix $X \succ 0$ such that $\langle A_j, X\rangle = b_j$ for all $j \in [m]$.
\item[(ii)] There exists $y \in \R^m$ such that $\Omega=\sum_{j=1}^m y_jA_j \succeq 0$, $\Omega\ne 0$, and $\Omega X=0$.
\end{itemize}
\end{lemma}

\begin{theorem}[{\cite[Theorems 2.21-2.22]{Tsirelson}}]
\label{theoExCorTsi}
Let $C\in \Cor(m,n)$, let $\{x_s\}$, $\{y_t\}$ be a $C$-system spanning $\R^r$, and let
$E=\Gram(x_1,\ldots,x_m,y_1,\ldots,y_n)\in\ME_{m+n}$.
\begin{itemize}
\item[(i)] If $C$ is an extreme point of $\Cor(m,n)$, then there exist nonnegative scalars $\lambda_1,\ldots,\lambda_m,$ $\mu_1,\ldots,\mu_n$, not all equal to zero, such that 
\begin{equation}\label{eqTsi}
\sum_{s=1}^m \lambda_s x_sx_s^\T = \sum_{t=1}^n \mu_t y_ty_t^\T .
\end{equation}
\item[(ii)] If  $E$ is an extreme point of $\ME_{m+n}$ and there exist strictly positive scalars $\lambda_1,\ldots,\lambda_m,\mu_1,\ldots,\mu_n$ for which relation (\ref{eqTsi}) holds, then $C$ is an extreme point of $\Cor(m,n)$.
\end{itemize}
\end{theorem}

\begin{proof}
(i) By assumption, $C$ is an extreme point of $\Cor(m,n)$, so by Lemma~\ref{lemExVec}~(ii) $E$ is the only feasible solution of the program (\ref{eqsdpP}).
As $E$ has rank $r<m+n$, it follows that the program (\ref{eqsdpP}) does not have a positive definite feasible solution.
Applying Lemma~\ref{theoFarkas} it follows  that there exists a nonzero matrix $\Omega$ that is feasible for the dual program (\ref{eqsdpD}) and satisfies   $\Omega E=0$.
This   gives:
$$\lambda_s x_s +\sum_{t\in T} W_{s,t}y_t=0 \ (s\in S),\ \ 
\mu_ty_t +\sum_{s\in S} W_{s,t} x_s=0\ (t\in T).$$
Since $\Omega\succeq 0$, the scalars $\lambda_s,\mu_t$ are nonnegative. We claim that they satisfy (\ref{eqTsi}). We multiply the left relation by $x_s^\T$ and the right one by $y_t^\T$ to obtain
$$\lambda_s x_s x_s^\T  +\sum_{t\in T} W_{s,t} y_t x_s^\T=0 \ (s\in S),\ \ 
\mu_ty_t y_t^\T +\sum_{s\in S} W_{s,t} x_s y_t^\T=0\ (t\in T).$$
Summing the left relation over $s\in S$,  and summing the right relation over $t\in T$ and taking the transpose, we get:
$$\sum_{s\in S} \lambda_s x_s x_s^\T= -\sum_{s\in S}\sum_{t\in T}W_{s,t}y_t x_s^\T= \sum_{t\in T} \mu_t y_t y_t^\T,$$
and thus (\ref{eqTsi}) holds. 

(ii) Assume that $E$ is an extreme point of $\ME_{m+n}$ and that there exist strictly positive scalars $\lambda_1,\ldots,\lambda_m,\mu_1,\ldots,\mu_n$ for which (\ref{eqTsi}) holds.
The key idea is to construct a matrix $\Omega$ that is optimal for the program (\ref{eqsdpD}) and has rank $m+n-r$, since then we can apply Theorem~\ref{theoLV} and conclude that $E$ is the only extension of $C$ in $\ME_{m+n}$. The construction of such a matrix $\Omega$ is analogous to the construction given in \cite{CG} for frameworks (see Theorem 4.3  and its proof), so we omit the details.
\end{proof}

\begin{corollary}[\cite{Tsirelson}]\label{corbound}
If $C$ is an extreme point of $\Cor(m,n)$, then
\begin{equation}\label{eqrankC}
{\rank (C)+1\choose 2} \le n+m-1.
\end{equation}
\end{corollary}
\begin{proof}
Let $x_1,\ldots,x_m,y_1,\ldots,y_n \in \R^r$, with $r = \rank(C)$, be a $C$-system spanning $\R^r$ and let $E$ be their Gram matrix.
As $E$ is an extreme point of $\ME_{m+n}$, it follows from relation (\ref{eqrkE}) that $\mathcal S^r$ is spanned by the $m+n$ matrices $x_ix_i^\T , y_jy_j^\T $ ($i\in S, j\in T$). Combining this with the identity (\ref{eqTsi}) this implies that $\mathcal S^r$ is spanned by a set of $m+n-1$ matrices and thus its dimension ${r+1\choose 2}$ is at most $m+n-1$.
\end{proof}
Our first construction in the next section provides instances where the bound (\ref{eqrankC}) is tight.

\subsection{Constructions of extreme bipartite correlation matrices}

We construct two families of extreme points of $\Cor(m,n)$, which we will use in Section~\ref{secfinal} to construct completely positive semidefinite matrices with exponentially large completely positive semidefinite rank. The first construction meets Tsirelson's bound and is used to prove Theorem~\ref{theomain}. The second construction will be used to recover one of the results of \cite{Slofstra11}.

\medskip
We begin with  constructing a  family of extreme points $C_1$ of $\Cor(m,n)$ with $\rank(C_1)=r$, $m=r$, and $n={r\choose 2}+1$, which thus shows that inequality \eqref{eqrankC} is tight. Such a family of bipartite correlation matrices can also be inferred from \cite{VertesiPal}, where the correlation matrices are obtained through analytical methods as optimal solutions of linear optimization problems over $\Cor(m, n)$. Instead, we use the sufficient conditions for extremality of bipartite correlations given above.

For this we will construct matrices $E_1,\Omega_1\in \mathcal S^{r+n}$ that satisfy the conditions of Theorem~\ref{theoLV}; that is, $E_1$ is an extreme point of $\ME_{r+n}$, $\Omega_1$ is positive semidefinite with support contained in the complete bipartite graph $K_{r,n}$, $\rank(E_1)=r$, $\rank(\Omega_1)=n$, and $\Omega_1 E_1=0$. Our construction of $\Omega_1$ is inspired by \cite{GP}, which studies the maximum possible rank of  extremal positive semidefinite matrices with a complete bipartite support.

Consider the matrix  $\widehat B\in \R^{r\times {r\choose 2}}$, whose columns are indexed by the pairs $(i,j)$ with $1\le i<j\le r$,  with entries $\smash{\widehat B_{i, (i,j)}}=1$, $\smash{\widehat B_{j,(i,j)}}=-1$ for $1\le i<j\le r$, and all other entries 0. We also consider the matrix $B \in \R^{r\times n}$ obtained by adjoining to $\smash{\widehat B}$ a last column equal to the all-ones vector $e$. Note that $BB^\T = r I_r$ and $\smash{\widehat B\widehat B^\T} = rI_r-J_r$.  Then define the following  matrices:
 $$\Omega'=\begin{pmatrix} nI_r & \sqrt n B\cr \sqrt n B^\T  & r I_n\end{pmatrix} \in \mathcal S^{r+n},\ \
 E'= \begin{pmatrix} I_r & -{\sqrt n\over r}B\cr -{\sqrt n \over r} B^\T  & {n\over r^2} B^\T B\end{pmatrix}\in \mathcal S^{r+n}.$$
Since 
\[
\Omega'=\begin{pmatrix} \sqrt{n\over r} B \cr \sqrt r I_n\end{pmatrix} \begin{pmatrix} \sqrt{n\over r} B  \cr \sqrt r I_n\end{pmatrix}^\T \quad \text{and} \quad
E' =\begin{pmatrix} I_r \cr -{\sqrt n\over r}B^\T \end{pmatrix} \begin{pmatrix} I_r \cr -{\sqrt n\over r}B^\T \end{pmatrix}^\T,
\]
it follows that  $\Omega'$ and $E'$ are positive semidefinite,  $\Omega'E'=0$,  $\rank(\Omega')=n$, and $\rank(E')=r$.
It suffices  now to modify the matrix $E'$ in order to get a matrix $E_1$ with an all-ones  diagonal.
For this, consider the diagonal matrix $$D=I_r\oplus {r\over \sqrt {2n}} I_{n-1} \oplus {\sqrt{r\over n}} I_1$$ and set 
$E_1 = D E'D$ and $\Omega_1 = D^{-1} \Omega'D^{-1}.$ 
Then $E_1$ has an all-ones diagonal,  it is in fact the Gram matrix of the vectors $e_1,\ldots,e_r$, $ (e_i-e_j)/\sqrt 2 $ (for $1\le i<j\le r$), and $(e_1+\ldots +e_r)/\sqrt r$, and thus $E_1$ is an extreme point of $\ME_{r+n}$.
Moreover,  $\Omega_1 E_1=0$, 
 $\rank E_1=r$,  and $\rank \Omega_1 = n$.
Therefore   the conditions of Theorem \ref{theoLV} are fulfilled and we can conclude that 
the matrix $C_1=\pi(E_1)$ is an extreme point of $\Cor(r,n)$. 
So we have shown part (i) in Theorem \ref{lemExtremePoint} below.
 
 \medskip 
Our second construction is inspired by the XOR-game considered by  Slofstra in  \cite[Section 7.2]{Slofstra11}. We construct a family of extreme points $C_2$ of $\Cor(m,n)$ with $\rank(C_2) = r-1$, $m =r$ and $n = {r \choose 2}$. Define the $(r+n)\times (r+n)$ matrices 
\[
\Omega_2=\left(\begin{matrix}
{\sqrt n} I_r &  \widehat B\cr  \widehat B^\T & {r\over \sqrt n} I_n\end{matrix}\right), \ \
E_2=\left(\begin{matrix} {1\over r-1}\widehat B\widehat B^\T  & -{r\over 2\sqrt n}\widehat B\cr -{r\over 2\sqrt n}\widehat B^\T & {1\over 2}\widehat B^\T\widehat B\end{matrix}\right).
\]
Note that 
\[
\Omega_2= \sqrt n \begin{pmatrix} \frac{1}{\sqrt r} \widehat B & \frac{1}{\sqrt r}  e \\ \sqrt{\frac{r}{n}} I_n & 0\end{pmatrix} \begin{pmatrix} \frac{1}{\sqrt r} \widehat B & \frac{1}{\sqrt r}  e \\ \sqrt{\frac{r}{n}} I_n & 0\end{pmatrix}^\T,\ \ E_2 =\begin{pmatrix} \frac{-1}{\sqrt{2n}} \widehat B \widehat B^\T \cr \frac{1}{\sqrt{2}} \widehat B^\T \end{pmatrix} \begin{pmatrix} \frac{-1}{\sqrt{2n}} \widehat B \widehat B^\T \cr \frac{1}{\sqrt{2}} \widehat B^\T \end{pmatrix}^\T,
\]
where we use that $\widehat B \widehat B^\T \widehat B = (r I_r - J_r) \widehat B = r \widehat B$. It follows that $\Omega_2$ and $E_2$ are positive semidefinite, $\rank(\Omega_2)= n+1$ and $\rank(E_2)=r-1$. 
Moreover, one can check that $\Omega_2 E_2 = 0$. 
In order to be able to apply Theorem~\ref{theoLV} it remains to verify that $E_2$ is an extreme point of $\mathcal E_{r+n}$. 
 
The above factorization of $E_2$ shows that it is the Gram matrix of the system of vectors in $\R^r$: 
$$\left\{u_k={1\over \sqrt{2n}}(e-re_k): k\in [r]\right\}
\cup \left\{v_{ij}={1\over \sqrt 2}(e_i-e_j): 1\le i<j\le r\right\}.$$
As the vectors $u_k,v_{ij}$ lie in $\R^r$ while $E_2$ has rank $r-1$ we need to consider another Gram representation of $E_2$ by vectors in $\R^{r-1}$. For this, let $Q$ be an $r\times r$ orthogonal matrix with columns $p_1,\ldots,p_r$ and $p_r=1/\sqrt{r}e$. Then the vectors $\smash{\{Q^\T u_k\} \cup \{Q^\T v_{ij}\}}$ form again a Gram representation of $E_2$. Furthermore, as all $u_k,v_{ij}$ are orthogonal to the vector $p_r$ it follows that the vectors $Q^\T u_k$ and $Q^\T v_{ij}$ are all orthogonal to $Q^\T p_r=e_r$. Hence $Q^\T u_k=(x_k,0)$ and $Q^\T v_{ij} = (y_{ij},0)$ for some vectors $x_k,y_{ij} \in \R^{r-1}$ which now provide a Gram representation of $E_2$ in $\R^{r-1}$.

In order to conclude that $E_2$ is an extreme point of $\mathcal E_{r+n}$ it suffices, by Theorem~\ref{theoLiTam}, to verify that the set $\{x_k x_k^\T\} \cup \{ y_{ij} y_{ij}^\T\}$ spans the whole space $\mathcal S^{r-1}$.  Equivalently, we must show that the set $\{Q^\T u_k u_k^\T Q \} \cup \{Q^\T v_{ij} v_{ij}^\T Q\}$ spans the subspace $\{R\oplus 0: R\in\mathcal S^{r-1}\}$ of $\mathcal S^r$, or, in other words, that the set $\{u_k u_k^\T\} \cup \{ v_{ij} v_{ij}^\T\}$ spans the subspace
\[
\mathcal M:=\{Q (R\oplus 0)Q^\T: R\in \mathcal S^{r-1}\}\subseteq \mathcal S^r.
\]
Observe that $\text{dim}(\mathcal M) = {r \choose 2}$. We also have that $\text{span} \{v_{ij} v_{ij}^\T: 1 \leq i < j \leq r\}$ is contained in
\[
\mathrm{span}( \{u_k u_k^\T: k \in [r]\} \cup \{v_{ij} v_{ij}^\T: 1 \leq i < j \leq r\}) \subseteq \mathcal M,
\]
and that
\[
\text{span} \{v_{ij} v_{ij}^\T: 1 \leq i < j \leq r\} = \text{span} \{ (e_i - e_j) (e_i -e_j)^\T: 1 \leq i < j \leq r\}
\] has dimension ${r \choose 2}$. Therefore, equality holds throughout: 
\[
\text{span}( \{u_k u_k^\T: k \in [r]\} \cup \{v_{ij} v_{ij}^\T: 1 \leq i < j \leq r\} )= \mathcal M,
\]
and thus $E_2$ is an extreme point of $\mathcal E_{r+n}$. 
 
This shows that the conditions of Theorem \ref{theoLV} are satisfied and we can conclude that 
the matrix $C_2=\pi(E_2)$ is an extreme point of $\Cor(r,n)$. So we have shown part (ii) in Theorem \ref{lemExtremePoint} below.

\begin{theorem} \label{lemExtremePoint}
Consider an integer $r\ge 1$ and let $e_1,\ldots,e_r$ denote the standard unit vectors in $\R^r$.
\begin{itemize}
\item[(i)]
There exists a matrix $C_1$ which is an extreme point of $\smash{C(r,{r\choose 2}+1)}$ and has rank $r$. We can take $C_1$ to be the matrix with columns $(e_i-e_j)/\sqrt 2$ (for $1\le i<j\le r$) and $(e_1+\ldots+e_r)/\sqrt r$.
\item[(ii)] 
There exists a matrix $C_2$ which is an extreme point of $\Cor(r,{r\choose 2})$ and has rank $r-1$.
We can take $C_2$ to be the matrix whose  columns are $- \sqrt{r/(2(r-1))}(e_i-e_j)$ for $1\le i<j\le r$.
\end{itemize}
\end{theorem}

We conclude this section  with explaining how our second construction permits to recover a lower bound of Slofstra \cite{Slofstra11}  for  the amount  of entanglement needed by any optimal quantum strategy for the XOR-game he considers in \cite[Section 7.2]{Slofstra11}.

The goal of an XOR-game is to find a quantum strategy with maximal winning probability, or, equivalently,  a strategy that maximizes the bias of the game. For an introduction to XOR-games we refer to, e.g., \cite{JopThesis,CHTW04}. An XOR-game is given by a game matrix, and the game presented in \cite[Section 7.2]{Slofstra11} has game matrix $\smash{\widehat B}$ as defined above.  An optimal quantum strategy corresponds to an optimal solution of the following optimization problem: 
\begin{equation} \label{slofstraopt1}
\max \{\langle \widehat B, C \rangle: C \in \Cor(m,n)\}.
\end{equation}
Slofstra \cite{Slofstra11} showed (using the notion of `solution algebra' of the game) that  any tensor operator representation of any optimal solution $C$ of (\ref{slofstraopt1}) has local dimension at least $2^{\lfloor (r-1)/2\rfloor}$ (see Section~\ref{secquantumcor} for the definition of a tensor operator representation). As we now point out  this can also be derived from Tsirelson's results using our treatment. 

 For this note first that problem (\ref{slofstraopt1}) is equivalent to 
 \begin{equation} \label{slofstraopt1b}
\min  \{ \langle \widehat B, C \rangle: C \in \Cor(m,n)\}
 \end{equation}
(since $C \in \Cor(m,n)$ if and only if  $-C \in \Cor(m,n)$). 
Problem (\ref{slofstraopt1b}) 
 is in turn equivalent to the following optimization problem over the elliptope:
\begin{equation} \label{slofstraopt}
\min  \{\langle \Omega_2, E \rangle: E \in \mathcal E_{m+n}\}, 
\end{equation}
with $\Omega_2$ being defined as above (since $E\in \mathcal E_{m+n}$ is optimal for (\ref{slofstraopt})  if and only if  $C=\pi(E)\in \Cor(m,n)$ is optimal for 
(\ref{slofstraopt1b})). 
As $\Omega_2$ is positive semidefinite and $\langle \Omega_2,E_2\rangle =0$, it follows that $E_2$ is optimal for (\ref{slofstraopt}) and thus $C_2=\pi(E_2)$ is optimal for (\ref{slofstraopt1b}).
Moreover, as $\rank(E_2)=m+n-\rank(\Omega_2)$ is the largest possible rank of an optimal solution of (\ref{slofstraopt}), it follows from a geometric property of semidefinite programming that $E_2$ must lie in the relative interior of the set of optimal solutions of (\ref{slofstraopt}).  This, combined with the fact that $E_2$ is an extreme point of  $\ME_{m+n}$, implies that $E_2$ is the unique optimal solution of (\ref{slofstraopt}) and thus $C_2$ is the unique optimal solution of (\ref{slofstraopt1b}). Finally, as $C_2$ is an extreme point of $\Cor(m,n)$ with rank $r-1$, we can conclude using Corollary~\ref{cortensordim} below 
that any tensor operator representation of $C_2$ uses local dimension at least $2^{\lfloor (r-1)/2\rfloor}$, and the same holds for the unique optimal solution $-C_2$ of \eqref{slofstraopt1}.

\section{Lower bounding the size of operator representations}
\label{secquantumcor}

We start with recalling, in Theorem~\ref{Tsir}, some  equivalent characterizations for bipartite correlations in terms of operator representations, due to Tsirelson. 
For this consider a matrix $C\in \R^{m\times n}$. We  say that  $C$ admits a \emph{tensor operator representation} if there exist an integer $d$ (the \emph{local dimension}), a unit vector $\psi \in \C^d \otimes \C^d$, and Hermitian $d \times d$ matrices $\{X_s\}_{s=1}^m$ and $\{Y_t\}_{t=1}^n$ with spectra contained in $[-1,1]$, such that $C_{s,t} = \psi^* (X_s \otimes Y_t) \psi$ for all $s$ and $t$. 

Moreover we say that $C$ admits a (finite dimensional) {\em  commuting operator representation} if there exist an integer $d$, a Hermitian positive semidefinite $d \times d$ matrix $W$ with $\mathrm{trace}(W) = 1$, and Hermitian $d \times d$ matrices $\{X_s\}$ and $\{Y_t\}$ with spectra contained in $[-1,1]$,  such that $X_s Y_t = Y_t X_s$ and $C_{s,t} = \Tr(X_s Y_t W)$ for all $s$ and $t$. A commuting operator representation is said to be \emph{pure} if $\rank(W) = 1$.

Existence of these various operator representations relies on using Clifford algebras. For an integer $r\ge 1$  the \emph{Clifford algebra} $\Clifford(r)$ of order $r$ can be defined as the universal $C^*$-algebra with Hermitian generators $a_1,\ldots,a_r$ and relations 
\begin{equation} \label{eqCliffordrelations}
a_i^2 = 1 \quad \text{and} \quad a_ia_j + a_ja_i = 0 \quad \text{for} \quad i \neq j.
\end{equation}
We call these relations the \emph{Clifford relations}. 
To represent the elements of $\Clifford(r)$ by matrices we can use the following map, which is a $*$-isomorphism onto its image:
\begin{equation} \label{eqClifford}
\varphi_r \colon \Clifford(r) \to \C^{2^{\lceil r/2 \rceil} \times 2^{\lceil r/2 \rceil}}, \, \varphi_r(a_i) = \begin{cases}
Z^{\otimes \frac{i-1}{2}} \otimes X \otimes I^{\otimes \lceil \frac{r}{2} \rceil - \frac{i+1}{2}} & \text{for $i$ odd},\\
Z^{\otimes \frac{i-2}{2}} \otimes Y \otimes I^{\otimes \lceil \frac{r}{2}\rceil - \frac{i}{2}} & \text{for $i$ even}.
\end{cases}
\end{equation}
Here we use the \emph{Pauli matrices}
\[
X = \begin{pmatrix} 0 & 1 \\ 1 & 0 \end{pmatrix}, \quad Y = \begin{pmatrix} 0 & -\mathbf{i} \\ \mathbf{i} & 0 \end{pmatrix}, \quad Z = \begin{pmatrix} 1 & 0 \\ 0 & -1 \end{pmatrix}.
\]
For even $r$ the representation $\varphi_r$ is  irreducible and thus $\Clifford(r)$ is isomorphic to the full matrix algebra with matrix size $2^{r/2}$. 
For odd $r$ the representation $\varphi_r$ decomposes as a direct sum of two irreducible representations, each of dimension $2^{\lfloor r/2 \rfloor}$. Therefore, if $X_1,\ldots,X_r$ is a set of Hermitian matrices  satisfying the relations $X_i^2 = I$ and $X_i X_j + X_j X_i = 0$ for $i \neq j$, then they must have size at least $2^{\lfloor r/2 \rfloor}$.
We refer to \cite[Section 5.4]{Procesi} for details about (representations of) Clifford algebras.

\begin{theorem}[\cite{Tsirelson:87}] \label{Tsir}
Let $C \in \R^{m \times n}$. The following statements are equivalent:
\begin{enumerate}
\item $C$ is a bipartite correlation.
\item $C$ admits a tensor operator representation.
\item $C$ admits a pure commuting operator representation. 
\item $C$ admits a commuting operator representation. 
\end{enumerate}
\end{theorem}
\begin{proof}
$(1) \Rightarrow (2)$ Let $C \in \Cor(m,n)$. That means there exist unit vectors $\{x_s\}$ and $\{y_t\}$ in $\R^r$, where $r = \rank(C)$, such that $C_{s,t} = \ip{x_s, y_t}$ for all $s$ and $t$. 
Set $d = 2^{\lfloor r/2 \rfloor}$ and define
\begin{equation*}
X_s = \sum_{i=1}^r (x_s)_i \pi(a_i), \quad Y_t = \sum_{i=1}^r (y_t)_i \pi(a_i)^\T,
\end{equation*}
where $\pi$ is an irreducible representation of $\mathcal C(r)$ by matrices of size $d$ (note that for $r$ even we could use the explicit representation $\varphi_r$).
With $\psi = \frac{1}{\sqrt{d}} \sum_{i=1}^d e_i \otimes e_i$ one can use the Clifford relations to derive the following identity (see for example \cite{JopThesis}):
\[
C_{s,t} = \ip{x_s, y_t} = \Tr(X_s Y_t^\T)/d = \psi^* (X_s \otimes Y_t) \psi \quad \text{for all} \quad s\in S, t \in T.
\]
The eigenvalues of the matrices $\pi(a_1),\ldots,\pi(a_r)$ lie in $\{-1,1\}$, and the Clifford relations \eqref{eqCliffordrelations} can be used to derive that the eigenvalues of $X_s$ and $Y_t$ also lie in $\{-1, 1\}$. Thus, $(\{X_s\}, \{Y_t\}, \psi)$ is a tensor operator representation of $C$.

$(2) \Rightarrow (3)$ If $(\{X_s\}, \{Y_t\}, \psi)$ is a tensor operator representation of $C$, then the operators $X_s \otimes I$ and $I \otimes Y_t$ commute, and by using the identity \[\psi^* (X_s \otimes Y_t) \psi = \Tr((X_s \otimes I)(I \otimes Y_t) \psi \psi^*)\]
we see that $(\{X_s \otimes I\}, \{I \otimes Y_y\}, \psi\psi^*)$ is a pure commuting operator representation.

$(3) \Rightarrow (4)$ This is immediate.

$(4) \Rightarrow (1)$
Suppose $(\{X_s\}, \{Y_t\}, W)$ is a commuting operator representation of $C$. Since $W$ is positive semidefinite and has trace $1$, there exist nonnegative scalars $\lambda_i$ and orthonormal unit vectors $\psi_i \in \C^d \otimes \C^d$ such that $W = \sum_i \lambda_i \psi_i\psi_i^*$ and $\sum_i \lambda_i = 1$. Then, 
\[
C_{s,t} = \Tr(X_s Y_t W) = \sum_i \lambda_i \Tr(X_s Y_t \psi_i \psi_i^*) = \sum_i \lambda_i \psi_i^* X_s Y_t \psi_i. 
\]
So, with
\[
x_s = \bigoplus_i \sqrt{\lambda_i} \begin{pmatrix} \Re(X_s \psi_i) \\ \Im(X_s \psi_i) \end{pmatrix} \quad \text{and} \quad y_t = \bigoplus_i \sqrt{\lambda_i} \begin{pmatrix} \Re(Y_t \psi_i) \\ \Im(Y_t \psi_i) \end{pmatrix}
\]
we have $C_{s,t} = \ip{x_s, y_t}$  and $\|x_s\|, \|y_s\| \leq 1$, and by using the observation in the proof of Lemma~\ref{lemExVec} we can extend the vectors $x_s$ and $y_t$ to unit vectors.
\end{proof}

\begin{corollary} \label{remTsi}
If C is a bipartite correlation matrix of rank $r$, then it admits a tensor operator representation in local dimension $2^{\lfloor r/2 \rfloor}$. If $C$ is a bipartite correlation matrix that admits a tensor operator representation in local dimension $d$, then it has a commuting operator representation by matrices of size $d^2$.
\end{corollary}

The remainder of this section is devoted to showing that there are bipartite correlation matrices for which every operator representation requires a large dimension. 

For this we need two more definitions. A commuting operator representation $(\{X_s\}, \{Y_t\} ,W)$ is \emph{nondegenerate} if there does not exist a projection matrix $P \neq I$ such that $PW\mkern-3.1muP = W$, $X_s P = PX_s$, and $Y_t P = P Y_t$ for all $s$ and $t$. 
It  is said to be \emph{Clifford} if there exist matrices $Q \in \R^{m \times m}$ and $R \in \R^{n \times n}$ with all-ones diagonals, such that
\begin{align*}
X_s X_{s'} + X_{s'} X_s &= 2 Q_{s,s'} I \quad \text{for all} \quad s,s' \in S,\\
Y_t Y_{t'} + Y_{t'} Y_t &= 2R_{t,t'} I \quad \text{for all} \quad t,t' \in T.
\end{align*}
We will use the following theorem from Tsirelson as crucial ingredient.
\begin{theorem} [{\cite[Theorem 3.1]{Tsirelson:87}}] \label{Thrm3.1}
If $C$ is an extreme point of $\Cor(m,n)$, then any nondegenerate commuting operator representation of $C$ is Clifford. 
\end{theorem}

We can now state and prove the main result of this section. 

\begin{theorem} \label{RankRep}
Let $C$ be an extreme point of $\Cor(m,n)$ and let $r = \rank(C)$. Every commuting operator representation of $C$ uses matrices of size at least $(2^{\lfloor r/2 \rfloor})^2$.
\end{theorem}
\begin{proof}
Let $(\{X_s\}, \{Y_t\}, W)$ be a commuting operator representation of $C$ where $X_s, Y_t$ and $W$ are matrices of size $d$. We will show  $d\ge (2^{\lfloor r/2 \rfloor})^2$. If this representation is degenerate, then there exists a projection matrix $P \neq I$ such that $PW\mkern-3.1muP = W$, $X_sP = PX_s$, and $Y_tP = PY_t$ for all $s$ and $t$. Let $P = \sum_{i=1}^k v_i v_i^*$ be its spectral decomposition, where the vectors $v_1,\ldots,v_k$ are orthonormal, and set $U = (v_1,\ldots,v_k)$. Then one can verify that $(\{ U^* X_s U \}, \{ U^* Y_s U \}, U^* W U)$ is a commuting operator representation of $C$ of smaller dimension. So, since we are proving a lower bound on the dimension, we may assume $(\{X_s\}, \{Y_t\}, W)$ to be a nondegenerate commuting operator representation.

By extremality of $C$ we may assume the operator representation is pure. Hence, there is a unit vector $\psi$ such that $W = \psi \psi^*$. This gives
\[
C_{s,t} = \Tr(X_s Y_t W) = \psi^* X_sY_t \psi = \ip{x_s, y_t},
\]
where
\[
x_s = \begin{pmatrix} \Re(X_s \psi) \\ \Im(X_s \psi) \end{pmatrix} \quad \text{and} \quad y_t = \begin{pmatrix} \Re(Y_t \psi) \\ \Im(Y_t \psi) \end{pmatrix}.
\]
These vectors $x_s$ and $y_t$ are unit vectors because $C$ is extreme (see the proof of Lemma~\ref{lemExVec}), and therefore, they form a $C$-system. 

By Theorem~\ref{Thrm3.1} the commuting operator representation $(\{X_s\}, \{Y_t\}, W)$ is Clifford. So, there exist matrices
$Q \in \R^{m \times m}$ and $R \in \R^{n \times n}$ with all-one diagonals such that
\begin{align*}
X_s X_{s'} + X_{s'} X_s &= 2Q_{s,s'} I \quad \text{for all} \quad s,s' \in S,\\
Y_t Y_{t'} + Y_{t'} Y_t &= 2R_{t,t'} I \quad \text{for all} \quad t,t' \in T.
\end{align*}
We show that $E$ is an extension to the elliptope of $C$, where 
\[
E = \begin{pmatrix} Q & C \\ C^\T & R \end{pmatrix} .
\]
For this, we have to show $Q_{s,s'} = \ip{x_s,x_{s'}}$ and $R_{t,t'} = \ip{y_t,y_{t'}}$. Indeed,
\begin{align*}
\ip{x_s,x_{s'}} + \ip{x_{s'},x_s} &= \Re\big(\psi^*X_s X_{s'} \psi + \psi^* X_{s'} X_s \psi\big) \\
&= \Re\big(\psi^* (X_s X_{s'} + X_{s'} X_s) \psi\big) \\
&= \Re\big(\psi^* ( 2Q_{s,s'} I ) \psi\big) = 2Q_{s,s'},
\end{align*}
and in the same way $\ip{y_t,y_{t'}} + \ip{y_{t'},y_t} = 2 R_{t,t'}$.

By Theorem \ref{theoExCor} the matrix $E$ is the unique extension of $C$ to the elliptope. Furthermore, Lemma~\ref{lemExVec} tells us that $\mathrm{rank}(Q) = \mathrm{rank}(R) = \mathrm{rank}(C) = r$. 

Consider the spectral decomposition $Q = \sum_{i=1}^r \alpha_i v_i v_i^*$, where the vectors $v_1,\ldots,v_r$ are orthonormal, and consider the algebra $\C\langle A_1,\ldots,A_r \rangle$, where
\[
A_i = \frac{1}{\sqrt{\alpha_i}} \sum_{s = 1}^m (v_i)_s X_s \quad \text{for} \quad i \in [r].
\]
We have 
\begin{align*}
A_i A_j + A_j A_i &= \frac{1}{\sqrt{\alpha_i\alpha_j}} \sum_{s,s' = 1}^m \left((v_i)_s (v_j)_{s'} X_s X_{s'} + (v_j)_s (v_i)_{s'} X_s X_{s'}\right)\\
&= \frac{1}{\sqrt{\alpha_i\alpha_j}} \sum_{s,s' = 1}^m (v_i)_s (v_j)_{s'} \left(X_s X_{s'} + X_{s'} X_s \right)\\
&= \frac{1}{\sqrt{\alpha_i\alpha_j}} \sum_{s,s' = 1}^m (v_i)_s (v_j)_{s'} 2Q_{s,s'} I 
= \frac{2}{\sqrt{\alpha_i\alpha_j}}  v_i^* Q v_j I = 2 \delta_{i,j} I,
\end{align*}
which means that we have the representation $\pi_A \colon \mathcal C(r) \to \C\langle A_1,\ldots,A_r \rangle$ defined by $\pi_A(a_i) = A_i$, where the $a_i$ are the generators of $\mathcal C(r)$. In the same way we can define matrices $B_1,\ldots,B_r$ by taking linear combinations of the matrices $Y_t$ so that we obtain the representation $\pi_B \colon \mathcal C(r) \to \C\langle B_1,\ldots, B_r \rangle$ defined by $\pi_B(a_i) = B_i$.

By assumption, the algebras $\C\langle X_1,\ldots,X_m \rangle$ and $\C\langle Y_1,\ldots, Y_n \rangle$ commute. This implies that the  algebras $\C\langle A_1,\ldots,A_r \rangle$ and $\C\langle B_1,\ldots,B_r \rangle$ also commute and that
$
\C\langle A_1,\ldots,A_r \rangle \C\langle B_1,\ldots,B_r \rangle
$
is an algebra. Moreover, we have 
\[
[\pi_A(a), \pi_B(b)] = \pi_A(a)\pi_B(b) - \pi_A(a)\pi_B(b) = 0 \quad \text{for all} \quad a,b \in \mathcal C(r).
\]
By the universal property of the tensor product of algebras (see, e.g., \cite[Proposition II.4.1]{Kas}), there exists a (unique) algebra homomorphism 
 $$\pi: \mathcal C(r) \otimes \mathcal C(r)\to  \C\langle A_1,\ldots,A_r \rangle \C\langle B_1,\ldots,B_r \rangle$$ such that 
 $\pi(a\otimes 1)=\pi_A(a)$ and $\pi(1\otimes a)=\pi_B(a)$ for all $a\in \mathcal C(r)$.
Moreover, each finite dimensional, irreducible representation of a tensor product of algebras is the tensor product of two irreducible representations of those algebras (see, e.g., \cite[Remark 2.27]{EGHLSVY11}). This means that each irreducible representation of $\mathcal C(r) \otimes \mathcal C(r)$ is the tensor product of two irreducible representations of $\mathcal C(r)$. Since irreducible representations of $\mathcal C(r)$ have size at least $2^{\lfloor r/2 \rfloor}$, it follows that irreducible representations of the tensor product $\mathcal{C}(r) \otimes \mathcal{C}(r)$ must have size at least $(2^{\lfloor r/2 \rfloor})^2$. Since $\pi$ is a representation of $\mathcal C(r) \otimes \mathcal C(r)$, this means that the matrices $A_i$ and $B_j$ must have size at least $(2^{\lfloor r/2 \rfloor})^2$, which shows $d \geq (2^{\lfloor r/2 \rfloor})^2$.
\end{proof}

\begin{corollary}\label{cortensordim}
Let $C$ be an extreme point of $\Cor(m,n)$ and let $r = \rank(C)$.  The minimum local dimension of a tensor operator representation of $C$ is $2^{\lfloor r/2 \rfloor}$.
\end{corollary}
\begin{proof}
The proof follows directly from Corollary~\ref{remTsi} and Theorem~\ref{RankRep}.
\end{proof}

\section{Matrices with high completely positive semidefinite rank}
\label{secfinal} 

In this section we  prove our main result and construct completely positive semidefinite matrices with exponentially large $\cpsd$. In order to do so we are going to use an additional link between bipartite correlations and quantum correlations, combined with the fact that quantum correlations arise as projections of completely positive semidefinite matrices. We start with recalling the facts that we need about quantum correlations.

Let $A$, $B$, $S$, and $T$ be finite sets. A function $p \colon A \times B \times S \times T \to [0,1]$ is called a \emph{quantum correlation}, realizable in {\em local dimension} $d$, if there exist a unit vector $\psi \in \C^d \otimes \C^d$ and Hermitian positive semidefinite $d \times d$ matrices $X_s^a$ ($s\in S$, $a\in A$) and $Y_t^b$ ($t\in T$, $b\in B$) satisfying the following two conditions:
\begin{equation}\label{eqqc1}
\sum_{a \in A} X_s^a = \sum_{b \in B} Y_t^b = I \quad \text{for all} \quad s \in S, t \in T,
\end{equation}
\begin{equation}\label{rqqc2}
p(a,b|s,t) = \psi^* (X_s^a \otimes Y_t^b) \psi \quad \text{for all} \quad a \in A, b \in B, s \in S, t \in T.
\end{equation}

The next theorem shows a link between quantum correlations and $\CSP$-matrices. This result can be found in \cite[Theorem 3.2]{Antonios:2015} (see also \cite{MR14}). This link allows us  to construct $\CSP$-matrices with large complex completely positive semidefinite rank by finding quantum correlations that cannot be realized in a small local dimension.

\begin{theorem} \label{thrm1}
A function $p \colon A \times B \times S \times T \to [0,1]$ is a quantum correlation that can be realized in local dimension $d$ if and only if there exists a completely positive semidefinite matrix $M$, with rows and columns indexed by the disjoint union $(A \times S) \sqcup (B \times T)$,
satisfying the following conditions:
\begin{equation} \label{eqrank}
\hcpsd(M) \leq d,
\end{equation} 
\begin{equation} \label{eqMprob}
M_{(a,s), (b,t)} = p(a,b|s,t) \quad \text{for all} \quad a\in A, b \in B, s \in S, t \in  T,
\end{equation}
and
\begin{equation} \label{eqMsums}
\sum_{a \in A, b \in B} M_{(a,s),(b,t)} = \sum_{a,a' \in A} M_{(a,s),(a',s')} = \sum_{b,b' \in B} M_{(b,t),(b',t')} = 1
\end{equation}
for all $s,s' \in S$ and $t,t'\in T$.
\end{theorem}

Next we show how to construct from a bipartite correlation $C\in \Cor(m,n)$ a quantum correlation $p$, with $|A|=|B|=2$ and $S=[m]$, $T=[n]$, having the property that the smallest local dimension in which $p$ can be realized is lower bounded by the smallest local dimension of a tensor representation of $C$.

\begin{lemma} \label{borp}
Let   $C \in \Cor(m, n)$ and assume $C$ admits a tensor operator representation in local dimension $d$, but does not admit a tensor operator representation in smaller dimension.
Then there exists a quantum correlation $p$ defined on  $\{0,1\} \times \{0,1\} \times [m] \times [n]$,
satisfying the relations
\begin{equation}\label{corp}
C(s,t) = p(0,0|s,t) + p(1,1|s,t) - p(0,1|s,t) - p(1,0|s,t) \text{ for } s \in [m], t \in [n],
\end{equation}
that can be realized in local dimension $d$, but cannot be realized in smaller dimension.
\end{lemma}

\begin{proof}
We first show the existence of a quantum correlation that satisfies \eqref{corp}. Let $C \in \Cor(m,n)$. By assumption there exists a unit vector $\psi \in \C^{d} \otimes \C^{d}$, and Hermitian $d \times d$ matrices $X_1,\ldots,X_m,Y_1,\ldots,Y_n$, whose spectra are contained in $[-1, 1]$, such that $C_{s,t} = \psi^* (X_s \otimes Y_t) \psi$ for all $s$ and $t$. We define the Hermitian positive semidefinite matrices 
\begin{equation}\label{eqXY}
X^a_s={I+(-1)^a X_s\over 2},\ Y^b_t={I+(-1)^b Y_t\over 2}\ \text{ for } a,b\in \{0,1\}.
\end{equation}
Using the fact that $X_s^0 + X_s^1 = Y_t^0 + Y_t^1 = I$, $X_s = X_s^0 - X_s^1$, and $Y_t = Y_t^0 - Y_t^1$, it follows that the function $p(a,b|s,t) = \psi^* (X_s^a \otimes Y_t^b) \psi$ is a quantum correlation that can be realized in local dimension $d$ and satisfies \eqref{corp}.

Assume that $p$ can be realized in dimension $k$, we show that $k\ge d$. As $p$ is realizable in dimension $k$  there exist a unit vector $\tilde\psi \in \C^k \otimes \C^k$ and Hermitian positive semidefinite $k \times k$ matrices $\{\tilde X_s^a\}$ and $\{\tilde Y_t^b\}$ such that 
\[
\sum_{a \in \{0,1\}} \tilde X_s^a = \sum_{b \in \{0,1\}} \tilde Y_t^b = I \quad \text{for all} \quad s \in S, t \in T,
\]
for which we have $p(a,b|s,t) = \tilde \psi^* ( \tilde X_s^a \otimes  \tilde Y_t^b)  \tilde\psi$. Observe that the spectrum of the operators $ \tilde X_s^a$ and $ \tilde Y_t^b$ is contained in $[0,1]$. We define $ \tilde X_s =  \tilde X_s^0 -  \tilde X_s^1,  \tilde Y_t =  \tilde Y_t^0- \tilde Y_t^1$. Then,  using \eqref{corp}, we can conclude  \[C_{s,t} =  \tilde \psi^*( \tilde X_s \otimes  \tilde Y_t) \tilde \psi.\] This means that $C$ has a tensor  operator representation in local dimension $k$ and thus, by the assumption of the lemma, $k \geq d$. 
\end{proof}

We can now prove our main theorem:

\begin{theorem*}
For each positive integer $k$, there exists a completely positive semidefinite matrix $M$ of size $4k^2 +2k +2$ with $\hcpsd(M) =  2^k$.
\end{theorem*}
\begin{proof}

Let $k$ be a positive integer, let $r = 2k$, and set $n  = \binom{r}{2}+1$. By Theorem~\ref{lemExtremePoint}(i) there exists an extreme point $C$ of $\Cor(r,n)$ with $\rank(C) = r$. Corollary~\ref{cortensordim}  tells us there exists a tensor operator representation of $C$ using  local dimension $d =2^{\lfloor r/2\rfloor}= 2^k$, and that there does not exist a smaller tensor operator representation.
Then, by Lemma~\ref{borp}, there exists a quantum correlation $p \colon \{0,1\} \times \{0,1\} \times [r] \times [n] \to [0, 1]$ that can be realized in local dimension $d$ and not in smaller dimension. Let $M$ be a completely positive semidefinite matrix constructed from $p$ as indicated in Theorem~\ref{thrm1}, so that  $\hcpsd(M) = d$ and the size of $M$ is $2r + 2n = r^2 + r + 2 = 4k^2 + 2k + 2$. 
\end{proof}

We note that by using Theorem~\ref{lemExtremePoint}(ii) we would get a matrix with the same completely positive semidefinite rank $2^k$, but with larger size $4k^2+6k+2$. Likewise, the result of \cite{Ji13} combined with Theorem \ref{thrm1} also leads to a matrix with the same completely positive semidefinite rank, but with larger size ($148k^2-58k$).  It is an open problem to find a family of completely positive semidefinite matrices where the ratio of the completely positive semidefinite rank to the matrix size is larger than in the above theorem. It is not possible to obtain such an improved family by the above method. Indeed, if $M$ is a completely positive semidefinite matrix with $
\hcpsd(M) = 2^k$, constructed from an extreme bipartite correlation matrix $C \in \Cor(m,n)$ as in the above theorem, then the size $2m+2n$ of $M$ is at least $4k^2 +2k+2$. To see this, note that, by Corollary~\ref{cortensordim} and the results in this section, $C$ has to have rank $2k$. Then, by Tsirelson's bound, $m+n-1 \geq {2k+1 \choose 2}$ and therefore $2m+2n \geq 4k^2+2k+2$.

\bigskip
\noindent {\bf Acknowledgements.} We are grateful to an anonymous referee for his/her careful reading and helpful comments, and for bringing the works \cite{Ji13,Slofstra11} to our attention.

\end{document}